\documentclass[12pt]{article}

\usepackage{amsfonts,amsmath,amsthm,latexsym,amssymb,fullpage,amscd}
\usepackage[mathscr]{eucal}
\usepackage{tikz}

\newcommand{\R}{\mathbb{R}}
\newtheorem{theorem}{Theorem}[section]
\newtheorem{lemma}[theorem]{Lemma}

\newtheorem{definition}{Definition}[section]

\title{Smoothing theorems for Radon transforms over hypersurfaces and related operators}

\author{Michael Greenblatt}
\date{\today}

\begin{document}
\maketitle
\begin{abstract} 
We extend the theorems of [G1] on $L^p$ to $L^p_s$ Sobolev improvement for translation invariant Radon and fractional singular Radon
transforms over hypersurfaces, proving $L^p$ to $L^q_s$ boundedness results for such operators. Here $q \geq p$ but $s$ can 
be positive, negative, or zero. For many such operators we will have a triangle $Z \subset (0,1) \times (0,1) 
\times {\mathbb R}$ such that one has $L^p$ to $L^q_{s}$ boundedness for 
$({1 \over p}, {1 \over q}, s)$ beneath $Z$, and in the case of Radon transforms one does not have  $L^p$ to $L^q_{s}$ boundedness for $({1 \over p}, {1 \over q}, s)$ above the plane containing $Z$, thereby providing a Sobolev space improvement result which is sharp up to endpoints for $({1 \over p}, {1 \over q})$ below $Z$. This triangle $Z$ intersects  the plane $\{(x_1,x_2,x_3): x_3 = 0\}$, and
therefore we also have an $L^p$ to $L^q$ improvement result that is also sharp up to endpoints for certain ranges of $p$ and $q$.
 
\end{abstract}

\section{ Introduction and theorem statements } 

As in [G1], we consider convolution operators with hypersurface measures on $\R^{n+1}$. Namely, we consider operators 
of the following form, where $\bf{x}$ denotes $(x_1,...,x_n)$ and $\bf{t}$ denotes $(t_1,...,t_n)$.
$$Tf({\bf x}, x_{n+1}) = \int_{\R^n} f({\bf x} - {\bf t}, x_{n+1} - S({\bf t})) K({\bf t})\,d{\bf t}\eqno (1.1)$$
Here $S({\bf t})$ is a real-analytic function on a neighborhood $U$ of the origin and $K({\bf t})$ is a function, supported in $U$, that
is $C^1$ on $\{{\bf t} \in U: t_i \neq 0$ for all $i\}$ and which satisfies the following estimates. Write ${\bf t} = ({\bf t}_1,...,{\bf t}_m)$,
where ${\bf t}_i$ denotes  $(t_{i1},...,t_{il_i})$ such that the various $t_{ij}$ variables comprise the whole list $t_1,...,t_n$. Then for some
$0 \leq \alpha_i < l_i$ and some $C > 0$ we assume the following.
$$|K({\bf t})| \leq C\prod_{k=1}^m |{\bf t}_k|^{-\alpha_k} \eqno (1.2a) $$
$$|\partial_{t_{ij}}K({\bf t})| \leq C {1 \over |t_{ij}|} \prod_{k=1}^m |{\bf t}_k|^{-\alpha_k} {\hskip 0.4 in} {\rm \,\,for\,\,all\,\,} i {\rm\,\, and\,\,} j \eqno (1.2b) $$
Operators satisfying $(1.1), (1.2a), (1.2b)$ are sometimes referred to as fractional Radon transforms or fractional singular 
Radon transforms. The case where each $\alpha_i = 0$ includes traditional Radon transform operators, by which we mean the operators where $K({\bf t})$ is
a $C^1$ function. By the translation and rotation
invariance properties of convolution operators, without loss of generality we may assume that
$$S(0,...,0) = 0 {\hskip 0.65 in} \nabla S(0,...,0) = (0,...,0) \eqno (1.3)$$
To avoid trivialities, we also assume $S$ is not identically zero. 

\noindent We will make use the following terminology and results from [G1].

\begin{definition}

Let $f({\bf t})$ be a real analytic function defined on a neighborhood of the origin in 
$\R^n$, and
let $f({\bf t}) = \sum_{\alpha} f_{\alpha}{\bf t}^{\alpha}$ denote the Taylor expansion of $f({\bf t})$ at the origin.
For any $\alpha$ for which $f_{\alpha} \neq 0$, let $Q_{\alpha}$ be the octant $\{{\bf t} \in \R^n: 
t_i \geq \alpha_i$ for all $i\}$. Then the {\it Newton polyhedron} $N(f)$ of $f({\bf t})$ is defined to be 
the convex hull of all $Q_{\alpha}$.  

\end{definition}

\begin{definition} 

Where $f({\bf t})$ is as in Definition 1.1, define $f^*({\bf t})$ by 
$$f^*({\bf t}) = \sum_{(v_1,...,v_n)\,\,a \,\,vertex \,\,of\,\,N(f)} |t_1|^{v_1}...|t_n|^{v_n} \eqno (1.4)$$
\end{definition}

\noindent By Lemma 2.1 of [G2], there is a neighborhood $V$ of the origin and 
a constant $C$ such that for all ${\bf t} \in V$ one has $|f({\bf t})| \leq Cf^*({\bf t})$. 

Let $d\mu$ denote the measure $\prod_{k=1}^m |{\bf t}_k|^{-\alpha_k}\,dm$, where $m$ denotes Lebesgue measure.
By Lemma 2.1 of [G1] there is an $r_0 > 0$, an $a_0> 0$, and an integer $d_0$ satisfying $0 \leq d_0 \leq n-1$, such that if $r < r_0$ then there are positive constants $b_r$ and $B_r$ such that for $0 < \epsilon < {1 \over 2}$ we have
$$b_r \epsilon^{a_0} |\ln \epsilon|^{d_0} < \mu(\{{\bf t} \in (0,r)^n:  S^*({\bf t}) < \epsilon\}) < B_r \epsilon^{a_0} |\ln\epsilon|^{d_0}\eqno (1.5)$$
In order to state the main theorem of [G1], we will also need the following definitions.

\begin{definition}Suppose $F$ is a compact face of the Newton polyhedron $N(f)$. Then
if $f({\bf t}) = \sum_{\alpha} f_{\alpha}{\bf t}^{\alpha}$ denotes the Taylor expansion of $f$ like above, 
define $f_F({\bf t}) = \sum_{\alpha \in F} f_{\alpha}{\bf t}^{\alpha}$.

\end{definition}

\begin{definition} For $f({\bf t})$ as above, we denote by $o(f)$ the maximum order of any zero of any $f_F({\bf t})$ on $(\R - \{0\})^n$. We take
$o(f) = 0$ if there are no such zeroes.

\end{definition}

\begin{definition} The Newton distance $d(f)$ is defined to be the minimal $t$ for which $(t,...,t)$ is in the Newton
polyhedron $N(f)$.

\end{definition}

\noindent The main theorem of [G1] is as follows.
\begin{theorem} Suppose $S({\bf t})$ is a real analytic function on a neighborhood of the origin satisfying $(1.3)$. Let $g = 
 \min(a_0, l_1 - \alpha_1,...,l_m - \alpha_m)$, where the $\alpha_i$ and $l_i$ are as in the beginning of this paper and $a_0$ is as in $(1.5)$. Then there is a neighborhood
 $V$ of the origin such that if $K({\bf t})$ is supported on $V$ and satisfies $(1.2a)-(1.2b)$ then the following hold.
 
 \noindent {\bf 1)} Let $A$ denote the open triangle with vertices $({1 \over 2}, {1 \over \max(o(S), 2)})$, $(0,0)$, and $(1,0)$, and  let $B = \{(x,y) \in A: 
 y < g\}$. Then $T$ is bounded from $L^p(\R^{n+1})$ to $L^p_{s}(\R^{n+1})$ if $({1 \over p}, s) \in B$. 
 
 \noindent {\bf 2)} Suppose $g < 1$, $K({\bf t})$ is nonnegative, and there exists a positive constant $C_0$ and a neighborhood $N_0$ of the origin such that $K({\bf t}) > C_0\prod_{k=1}^m |{\bf t}_k|^{-\alpha_k}$ on $\{{\bf t} \in N_0: t_i \neq 0$ for all $i\}$. Then if $1 < p < \infty$ and 
 $T$ is bounded from $L^p(\R^{n+1})$ to $L^p_{s}(\R^{n+1})$ we must have $s \leq g$.
\end{theorem}

Observe that when $g < {1 \over \max(o(S), 2)}$, the two parts of Theorem 1.1 combined say that for ${1 \over p} \in ( {\max(o(S), 2) \over 2}g, 1 -  {\max(o(S), 2) \over 2}g)$, the amount of $L^p$ 
Sobolev smoothing given by part 1, $g$ derivatives, is optimal except possibly missing the endpoint $s = g$. When $g = {1 \over \max(o(S), 2)}$ the same is true for $p = 2$. 

Some motivation for the index $g$ in Theorem 1.1 is as follows. Let $\nu$ denote the surface measure of $S$, weighted by $\prod_{k=1}^m |{\bf x}_k|^{-\alpha_k}$. If we are in a situation where $a_0 \leq {1 \over \max(o(S), 2)}$, the Newton polyhedron of $S(x)$ controls the decay 
rate of the Fourier transform $\hat{\nu}
(\lambda)$ in the $(0,...,0,1)$ direction, and in this direction the decay rate has a bound of $C|\lambda|^{-a_0 + \epsilon}$ for any $\epsilon > 0$.
This can be shown a minor variation on the arguments of [V] or [G2].
On the other hand, in any $(0,...,0,1,0,...,0)$ direction, a straightforward calculation shows that $\hat{\nu}(\lambda)$ decays at the rate of
$|\lambda|^{-{(l_i - \alpha_i)}}$, where $\alpha_i$ is such that this direction is one of the $t_{ij}$ directions. 

It can then be shown that in any other 
``diagonal'' direction, the Fourier transform decays at a rate no worse than the minimum of the above decay rates. Consequently, since $g = \min(a_0, l_1 - \alpha_1,...,l_m - \alpha_m)$, $g$ is the slowest possible decay rate of $\hat{\nu}$ in any direction. Given that
the $L^2$ Sobolev space improvement for $T$ is the largest exponent $\delta$ for which one has $|\hat{\nu}(\lambda)| \leq C|\lambda|^{-\delta}$, 
the $L^2$
case of Theorem 1.1 says that if $g < {1 \over \max(o(S), 2)}$, then up to endpoints one has such an estimate with $\delta= g$. In other words, the 
directional Fourier transform decay rates hold with a constant that is uniform over all directions. Furthermore, the statement of Theorem 1.1 gives that 
one has the same level of $L^p$ Sobolev improvement for $p$ in an interval containing $2$.

The following is the Sobolev space estimate we will use in our interpolation with the boundedness results of Theorem 1.1. We will be using it for $p$ approaching
$1$ and for $q$ tending to infinity.

\begin{theorem} For any $1 < p < q < \infty$ and any $\gamma > 1 + \sum_{i = 1}^m \alpha_i$,
the operator $T$ is bounded from $L^p(\R^{n+1})$ to $L^q_{-\gamma}(\R^{n+1})$.

\end{theorem}

Let $k =  1 + \sum_{i = 1}^m \alpha_i$. Observe that the plane $P$ in 3-space containing the line $\{(x,y,z): x = y, z = g\}$ and the point $(1,0, -k)$ has equation
$(g + k)(x - y) + z = g$.  Thus letting $p$ approach $1$ and $q$ approach infinity in Theorem 1.2 and interpolating with Theorem 1.1 gives the following,
keeping in mind that if $s_1 < s_2$ then $L^q_{s_2}(\R^{n+1}) \subset L^q_{s_1}(\R^{n+1})$ continuously for any $1 < q < \infty$. 

\begin{theorem} 

\
\

There is a neighborhood
 $V$ of the origin such that if $K({\bf t})$ is supported on $V$ and satisfies $(1.2a)-(1.2b)$ then the following hold.
 
Suppose $g < {1 \over \max(o(S),2)}$. Let $P$ denote the plane with equation $(g + k)(x - y) + z = g$, and let 
$Z$ be the closed triangle in $P$ whose vertices are $({\max(o(S), 2) \over 2}g, {\max(o(S), 2) \over 2}g, g), (1 - {\max(o(S), 2) \over 2}g, 1 - {\max(o(S), 2) \over 2}g, g)$, and $(1,0, - k)$. Then if $({1 \over p}, {1 \over q}, s)$ is such that there is a $t > s$ with $({1 \over p}, {1 \over q}, t)$ in the 
interior of $Z$, then
$T$ is bounded from $L^p(\R^{n+1})$ to $L^q_s(\R^{n+1})$. 

Suppose $g \geq {1 \over \max(o(S),2)}$. Let $L$ denote the open line segment joining $({1 \over 2}, {1 \over 2}, {1 \over \max(o(S),2)})$ with the point 
$(1,0,-k)$. Then if $({1 \over p}, {1 \over q}, s)$ is such that there is a $t > s$ with $({1 \over p}, {1 \over q}, t) \in L$, then
$T$ is bounded from $L^p(\R^{n+1})$ to $L^q_s(\R^{n+1})$. 

\end{theorem}

The triangle $Z$ can be visualized as follows. The segment from $({\max(o(S), 2) \over 2}g, {\max(o(S), 2) \over 2}g, g)$ to 
$(1 - {\max(o(S), 2) \over 2}g, 1 - {\max(o(S), 2) \over 2}g, g)$ is a line segment above the line $y = x$, at fixed height $z = g$, which is symmetric
about the midpoint $(1/2,1/2,g)$. The trangle $Z$ is then the convex hull of this segment and the point $(1,0,-k)$ that is below the lower-rightmost 
point in the square $[0,1] \times [0,1]$.

We can interpolate Theorem 1.3 with the trivial $L^p$ to $L^p$ estimates for $1 < p < \infty$ to obtain a larger region of Sobolev space boundedness. This
can be described as follows.

\begin{theorem} 

\
\

There is a neighborhood
 $V$ of the origin such that if $K({\bf t})$ is supported on $V$ and satisfies $(1.2a)-(1.2b)$ then the following hold.
 
Suppose $g < {1 \over \max(o(S),2)}$. Let $Z_1$ be the closed triangle with vertices $(0,0,0)$, $(1,0, - k)$, and $({\max(o(S), 2) \over 2}g, {\max(o(S), 2) \over 2}g, g)$, and let $Z_2$ be the closed triangle
with vertices $(1,1,0)$, $(1,0,-k)$, and $(1 - {\max(o(S), 2) \over 2}g, 1 - {\max(o(S), 2) \over 2}g, g)$. If $({1 \over p}, {1 \over q}, s)$ is such that there is a $t > s$ with $({1 \over p}, {1 \over q}, t)$ in the interior of $Z \cup Z_1 \cup Z_2$, then $T$ is bounded from $L^p(\R^{n+1})$ to $L^q_s(\R^{n+1})$. 

Suppose $g \geq {1 \over \max(o(S),2)}$. Let $Z_3$ be the closed triangle with vertices $(0,0,0)$, $(1,0,-k)$, and $({1 \over 2}, {1 \over 2}, {1 \over \max(o(S),2)})$ and let $Z_4$  be the closed triangle with vertices $(1,1,0)$, $(1,0,-k)$, and $({1 \over 2}, {1 \over 2}, {1 \over \max(o(S),2)})$. If $({1 \over p}, {1 \over q}, s)$ is such that there is a $t > s$ with $({1 \over p}, {1 \over q}, t)$ in the interior of $Z_3 \cup Z_4$, then $T$ is bounded from $L^p(\R^{n+1})$ to $L^q_s(\R^{n+1})$. 

\end{theorem}

So in Theorem 1.4, $Z_1$ is the convex hull of the left side of $Z$ with $(0,0,0)$ and $Z_2$ is  the convex hull of the right side of $Z$ with $(1,1,0)$, so 
that $Z_1$ and $Z_2$ are symmetric about the plane $x + y = 1$. Similarly, $Z_3$ is the convex hull of $(0,0,0)$ and the line segment from $({1 \over 2}, {1 \over 2}, {1 \over \max(o(S),2)})$ to $(1,0,-k)$ on the plane $x + y = 1$, and $Z_4$ is the convex hull of $(1,1,0)$ and the line segment from 
$({1 \over 2}, {1 \over 2}, {1 \over \max(o(S),2)})$ to $(1,0,-k)$. Again, $Z_3$ and $Z_4$ are symmetric about the plane 
$x + y = 1$, this time with a common edge on this plane.

The following theorem tells us when Theorem 1.3 gives the best possible amount of Sobolev smoothing, up to endpoints. Since Theorem 1.1 can only be
sharp up to endpoints in situations where $g \leq {1 \over \max(o(S),2)}$ this is the only situation when we can hope for such a result. In the following
theorem we will see that if each $\alpha_i  = 0$, such as in the case of (nonsingular) Radon transforms, if $g \leq {1 \over \max(o(S),2)}$  one never gets a
$({1 \over p}, {1 \over q}, s)$ boundedness theorem above the plane $P$ for any $1 <p,q < \infty$ (including when $p > q$.) Thus if $g < {1 \over \max(o(S),2)}$, Theorem 1.3 gives the optimal $s$ up to endpoints for $({1 \over p}, {1 \over q})$ beneath the triangle $Z$, and when $g = {1 \over \max(o(S),2)}$ Theorem 1.3 gives the optimal  $s$ 
up to endpoints for $({1 \over p}, {1 \over q})$ beneath the open line segment joining $(1,0, - k)$ and $({1 \over 2}, {1 \over 2}, {1 \over \max(o(S),2)})$.

\begin{theorem}

Suppose $g \leq {1 \over \max(o(S),2)}$ and $\alpha_i = 0$ for all $i$, such as in the case of (nonsingular) Radon transforms.  Suppose further that there is a $C_1 > 0$ and a 
neighborhood $N_0$ of the origin such that $K({\bf t}) > C_1$ on $N_0$. Then for any $1 < p, q < \infty$, 
if $({1 \over p}, {1 \over q}, s)$ is 
such that there is a $t < s$ with $({1 \over p}, {1 \over q}, t)$ on the plane $P$, then $T$ is not bounded from $L^p(\R^{n+1})$ to $L^q_s(\R^{n+1})$.

\end{theorem}

\noindent {\bf Extensions.}

Observe that by the translation invariance of $T$, whenever one 
has an $L^p(\R^{n+1})$ to $L^q_s(\R^{n+1})$ boundedness theorem, for any $b \in \R$ one also has the corresponding $L^p_b(\R^{n+1})$ to
 $L^q_{b+s}(\R^{n+1})$ boundedness theorem. 
 
The Sobolev embedding theorem can sometimes be used to extend the range of boundedness in Theorems 1.3 and 1.4 if $k$ is sufficiently close to its maximum 
possible value of $n+1$. Namely, it turns out that if $k > n + 1 - 2g$ in the case where $g < {1 \over \max(o(S),2)}$, or if 
$k > n+1 - {2 \over \max(o(S), 2)}$ in the case when $g \geq {1 \over \max(o(S),2)}$, one can sometimes extend Theorems 1.3 and 1.4 beyond
$Z_1 \cup  Z_2$ or $Z_3 \cup Z_4$ respectively in this fashion. If $k > n + 1 - g$  and  $g < {1 \over \max(o(S),2)}$, then one can also 
sometimes extend beyond the triangle $Z$. The sharpness theorem, Theorem 1.5, will be false in the latter situations.

\section{Some background} 

There has been quite a bit of work done on the boundedness properties of Radon transforms and fractional singular Radon transforms on function spaces, so we focus our attention on 
Sobolev space improvement and $L^p$ to $L^q$ improvement results for Radon transforms over hypersurfaces. For curves in $\R^2$, [S] provides comprehensive $L^p_{\alpha}$ to 
$L^q_{\beta}$ boundedness results for Radon transforms that are sharp up to endpoints. These results include general non-translation invariant operators.

For translation invariant Radon transforms, $L^2$ to $L^2_{\beta}$ Sobolev space improvement is equivalent to a surface measure Fourier transform 
decay rate estimate. When $n = 2$, the stability 
theorems of Karpushkin [Ka1] [Ka2] combined with [V] give such sharp decay rate results, again for the case of (nonsingular) Radon transforms.
 For situations where not all $\alpha_i$ are zero, the author has some results  [G3] [G4] in this area.

For higher dimensional hypersurfaces, in addition to the above-mentioned [G1], it follows from [St] that if the density functions are singular enough
 in the sense that the $\alpha_i$ are close enough to $l_i$, then there will be an interval containing 2 on which sharp $L^p$  to $L^p_{\beta}$ Sobolev smoothing holds. This extends the author's paper [G5]. We also mention the paper [Cu] which deals with $L^p$  to $L^p_{\beta}$ improvement for
fractional singular Radon transforms where the surface is relatively nondegenerate.

For specifically $L^p$ to $L^q$ improvement for Radon transforms over hypersurfaces, there have been a number of other results for 
Radon transforms. 
The situation where the  $S({\bf t})$ is a homogeneous or mixed homogeneous function has been considered in [FGU1] [FGU2] [DZ]. Convex surfaces were
considered in [ISaS]. Also, there have been papers considering weighted Radon transforms, where instead of singular $K({\bf t})$ as in this paper one considers
surfaces damped by a bounded $K({\bf t})$ with zeroes on a set chosen to be natural for the surfaces at hand. We mention [Gr] and [O] as examples of such
results.

\section{Examples}

\noindent {\bf Example 1.}

We consider the case of curves in two dimensions. So $S(t) = ct^l + O(t^{l+1})$ for some nonzero $c$ and some $l \geq 2$. Here $m = 1$, and there
is one $\alpha_k$ in $(1.2a)-(1.2b)$ which we denote by simply $\alpha$, where $0 < \alpha < 1$. Then $S^*(t) =t^l$. The index  $a_0$ of $(1.5)$ is the exponent  of $\epsilon$ in the measure of $\int_0^{\epsilon^{1 \over l}} t^{-\alpha}\,dt$ or $a_0 = {1 - \alpha \over l}$.
The Newton polyhedron $N(S)$ here has the one vertex $l$, and there is one polynomial $S_F(t)$ as in Definition 1.4,
given by $t^l$. Hence $o(S) = 0$ here. Thus the quantity $\max(o(S),2)$ in Theorem 1.1 is just 2, and the upper vertex of $A$ is $(1/2, 1/2)$. The quantity
$g$ of Theorem 1.1 is then given by $\min({1 - \alpha \over l}, 1 - \alpha) = {1 - \alpha \over l}$.

Looking at what Theorem 1.4 says here, we see that $Z_1$ has vertices $(0,0,0)$, $(1,0,-1 - \alpha)$, and $({1 - \alpha \over l},{1 - \alpha \over l},{1 - \alpha \over l})$, and $Z_2$ has vertices $(1,1,0)$, 
$(1,0, -1 - \alpha)$, and $(1 - {1 - \alpha \over l}, 1-{1 - \alpha \over l}, {1 - \alpha \over l})$. The triangle $Z$ has vertices $({1 - \alpha \over l},
{1 - \alpha \over l}, {1 - \alpha \over l})$, $(1- {1 - \alpha \over l}, 1 - {1 - \alpha \over l}, {1 - \alpha \over l})$, and $(1,0, - 1 - \alpha)$. (In the case
where $\alpha = 0$ and $l = 2$ the triangle $Z$ reduces to a line and we are in the second case of Theorem 1.4.) Theorem 1.4 then says that one
has $L^p$ to $L^q_{s}$ boundedness for $({1 \over p}, {1 \over q}, s)$ below the interior of $Z \cup Z_1 \cup Z_2$.

To specify the above to (nonsingular) Radon transforms one inserts $\alpha = 0$ into the above. Specifying further to $L^p$ to $L^q$ estimates, when
$l > 2$ we look at the intersection of $Z$ with the $x$-$y$ plane. Observe that $X_1 = {l \over l + 1} ({1 \over l}, {1 \over l}, {1 \over l}) + { 1 \over l + 1}(1,0,-1)$ and 
$X_2 =  {l \over l + 1} (1 - {1 \over l}, 1 - {1 \over l}, {1 \over l}) + { 1 \over l + 1}(1,0,-1)$ have third coordinate zero. As a result these two points will be 
on the intersection of $Z$ with the $x$-$y$ plane. Note that $X_1 = ({2 \over l + 1}, {1 \over l + 1}, 0)$ and $X_2 = ({l \over l + 1}, {l - 1 \over l + 1},0)$.
Thus we have $L^p$ to $L^q$ boundedness for $({1 \over p}, {1 \over q})$ in the interior of trapezoid with vertices $(0,0)$,  $({2 \over l + 1}, {1 \over l + 1})$, $({l \over l + 1}, {l - 1 \over l + 1})$, and $(1,1)$. In the case that $l = 2$ this reduces to the triangle with vertices $(0,0), ({2 \over 3}, {1 \over 3})$,
and $(1,1)$ and the second part of Theorem 1.4 gives $L^p$ to $L^q$ boundedness for  $({1 \over p}, {1 \over q})$ in the interior of triangle.
Theorem 1.5 then says that one does not have $L^p$ to $L^q$ boundedness for $({1 \over p}, {1 \over q})$ below the line containing the segment joining $({2 \over l + 1}, {1 \over l + 1})$ to $({l \over l + 1}, {l - 1 \over l + 1})$, namely the line $y = x - { 1\over l + 1}$. In fact, it follows from [S] that the trapezoid above is optimal
up to endpoints, but this requires an additional argument.

\noindent {\bf Example 2.} 

We move to the situation where $n \geq 2$ and consider the situation where each $\alpha_i = 0$, such as in the case of (nonsingular) Radon transforms, 
and where the order of each zero of each $S_F({\bf t})$ on $(\R - \{0\})^n$ is at most two. This includes the situation where the Newton polyhedron 
of $S$ is nondegenerate in the sense of Varchenko [V] and various other papers. Then as in the previous example, $\max(o(S),2) = 2$. By [V] the
quantity $a_0$ is given by ${1 \over d(S)}$, where $d(S)$ is the Newton distance of $S$ as in Definition 1.5. Since the stronger first part of Theorem 1.4 holds when $g < {1 \over 2}$ here, we focus our
attention on the situation where $d(S) > 2$ and therefore $g  = {1 \over d(S)} < {1 \over 2}$. Also, since each $\alpha_i = 0$, the 
quantity $k$ of Theorems 1.3 and 1.4 is just $1$.

In the situation at hand, $Z$ has vertices $({1 \over d(S)}, {1 \over d(S)},{1 \over d(S)})$, $(1 - {1 \over d(S)}, 1 - {1 \over d(S)}, {1 \over d(S)})$, 
and $(1,0,-1)$. The triangle $Z_1$ has vertices $(0,0,0)$, $(1,0,-1)$, and $({1 \over d(S)}, {1 \over d(S)},{1 \over d(S)})$, and the triangle $Z_2$ has 
vertices $(1,1,0)$, $(1,0,-1)$, and $(1 - {1 \over d(S)}, 1 - {1 \over d(S)}, {1 \over d(S)})$.  Theorem 1.4 gives 
$L^p$ to $L^q_{s}$ boundedness for $({1 \over p}, {1 \over q}, s)$ below the interior of $Z \cup Z_1 \cup Z_2$, and since each $\alpha_i = 0$, Theorem 1.5 says one cannot get 
$L^p$ to $L^q_{s}$ boundedness for $({1 \over p}, {1 \over q}, s)$ above the plane containing $Z$.

The intersection of $Z$ with the $x$-$y$ plane can be computed
to be the line segment joining $({2 \over d(S) + 1}, {1 \over d(S) + 1}, 0)$ and $({d(S)  \over d(S) + 1}, {d(S) - 1 \over d(S) + 1}, 0)$. Thus we have 
$L^p$ to $L^q$ boundedness  for $({1 \over p}, {1 \over q})$ in the interior of the trapezoid with vertices $(0,0), ({2 \over d(S) + 1}, {1 \over d(S) + 1})$, 
$({d(S)  \over d(S) + 1}, {d(S) - 1 \over d(S) + 1})$, and $(1,1)$. In the case where $d(S)$ is exactly two, similar to the previous example the second part
 of Theorem 1.4 gives $L^p$ to $L^q$ boundedness in the interior of the triangle with vertices $(0,0), ({2 \over 3}, {1 \over 3})$,
and $(1,1)$. Theorem 1.5 then says that one does not have $L^p$ to $L^q$ boundedness for $({1 \over p}, {1 \over q})$ below the line containing
the segment joining 
$({2 \over d(S) + 1}, {1 \over d(S) + 1})$ to $({d(S)  \over d(S) + 1}, {d(S) - 1 \over d(S) + 1})$, which is the line $y = x - { 1\over d(S) + 1}$.

\
\

\noindent {\bf Example 3.} 

Suppose now that each $l_i = 1$ so that each ${\bf t}_i$ is one dimensional. Like in the previous example we assume that
$o(S) \leq 2$. Then the quantities $a_0$ and $d_0$ of $(1.5)$ are 
defined by the condition that if $r$ is small enough there exist constants $b_r$ and $B_r$ such that for $0 < \epsilon < {1 \over 2}$ we have
$$b_r \epsilon^{a_0} |\ln \epsilon|^{d_0} < \int_{\{{\bf t} \in (0,r)^n:  S^*({\bf t}) < \epsilon\}} \prod_{i=1}^n t_i^{-\alpha_i}\,d{\bf t} < B_r \epsilon^{a_0} |\ln\epsilon|^{d_0}\eqno (3.1)$$
We change variables $t_i = u_i^{1 \over 1 - \alpha_i}$ in $(3.1)$, so that up to a constant, $t_i^{\alpha_i}\,dt_i = du_i$. Then $(3.1)$ becomes
$$b_r' \epsilon^{a_0} |\ln \epsilon|^{d_0} < m(\{{\bf u} \in (0,r)^n:  S^*(u_1^{1 \over 1 -  \alpha_1},...,u_n^{1 \over 1 - \alpha_n}) < \epsilon\})  < B_r' \epsilon^{a_0} |\ln\epsilon|^{d_0}\eqno (3.2)$$
Here $m$ denotes Lebesgue measure. Let $R(u_1,...,u_n) = S({\rm sgn}(u_1)|u_1|^{1 \over 1 - \alpha_1},..., {\rm sgn}(u_n)|u_n|^{1 \over 1 - \alpha_n})$.
Observe that  $(\beta_1,...,\beta_n) \rightarrow ({1 \over 1 - \alpha_1}\beta_1,..., {1 \over 1 - \alpha_n}\beta_n)$
takes the Newton polyhedron $N(S)$ to the Newton polyhedron $N(R)$ with faces getting mapped to corresponding faces.
Since $(u_1,...,u_n) \rightarrow (u_1^{1 \over 1 -  \alpha_1},...,u_n^{1 \over 1 - \alpha_n})$ is a diffeomorphism on $(0,r)^n$, the maximum order of
a zero of a given $S_F(t)$ on $(\R - \{0\})^n$ is the same as the maximum order of the corresponding $R_{\bar{F}}(t)$ on $(\R - \{0\})^n$ , where
$\bar{F}$ is the face of $N(R)$ corresponding to $F$. Hence $o(R) = o(S)$, which by our assumptions is at most $2$.

Next, observe that $R^*({\bf u}) = S^*(|u_1|^{1 \over 1 -  \alpha_1},...,|u_n|^{1 \over 1 - \alpha_n})$, where we define $R^*({\bf u})$ analogously to
$(1.4)$. Thus by an immediate modification of the argument for $S({\bf t})$, the quantity $a_0$ of $(3.2)$ is given by ${1 \over d(R)}$.  So the 
quantity $g$ of our theorems is given by  $g = \min({1 \over d(R)}, 1 - \alpha_1,...,1 - \alpha_n)$. 

Most of the time,
the quantity ${1 \over d(R)}$ will be smaller than each $1- \alpha_i$. For example, suppose each
 $\alpha_i = \alpha$ for
some $0 < \alpha < 1$. Then since the terms of $S^*({\bf u})$ have degree at least $2$, one has that $d(R) > C(1 - \alpha)^{-2}$ for some constant $C$.
Hence ${1 \over d(R)} < C'(1 - \alpha)^2$, so that if $\alpha$ is close enough to 1 then we have $g = \min({1 \over d(R)}, 1 - \alpha) = {1 \over d(R)}$.

Motivated by the above, we now add the assumption that $g = {1 \over d(R)}$. We also add the assumption $d(R) > 2$ so that we are in the 
$g < {1 \over 2}$ case where the results are strongest. Since we are also assuming that $o(R) = o(S) \leq 2$, we are in
the setting of the first parts of Theorems 1.3-1.4. The plane $P$  has equation $({1 \over d(R)} + 1 + \sum_{i=1}^n \alpha_i)(x - y) + z = 
{1 \over d(R)}$, and the triangles $Z$, $Z_1$, and $Z_2$ are determined as in example 2, if we replace $d(S)$ in that example by $d(R)$ and 
the vertex $(1,0,-1)$ by $(1,0, -k)$ where $k = 1 + \sum_{i=1}^n \alpha_i$. So 
$Z$ has vertices $({1 \over d(R)}, {1 \over d(R)},{1 \over d(R)})$, $(1 - {1 \over d(R)}, 1 - {1 \over d(R)}, {1 \over d(R)})$, 
and $(1,0,-k)$. The triangle $Z_1$ has vertices $(0,0,0)$, $(1,0,-k)$, and $({1 \over d(R)}, {1 \over d(R)},{1 \over d(R)})$, and the triangle $Z_2$ has 
vertices $(1,1,0)$, $(1,0,-k)$, and $(1 - {1 \over d(R)}, 1 - {1 \over d(R)}, {1 \over d(R)})$. 

Like in example 2, we have $L^p$ to $L^q_{s}$ boundedness for $({1 \over p}, {1 \over q}, s)$ below the interior of $Z \cup Z_1 \cup Z_2$. However unlike in example 2, if any $\alpha_i \neq 0$ the sharpness statement of Theorem 1.5 cannot be assumed to hold. However, the sharpness statements of [G1]  
for $L^p$ to $L^p_s$ boundedness tell us that one can not have an $L^p$ to
$L^p_s$ boundedness theorem when $s$ is above the line connecting the vertices $({1 \over d(R)}, {1 \over d(R)},{1 \over d(R)})$ and $(1 - {1 \over d(R)}, 1 - {1 \over d(R)}, {1 \over d(R)})$ of $Z$, namely the line $\{(t,t,{1 \over d(R)}): 0 < t < 1\}$. 

\section{The proof of Theorem 1.2}

The inclusion relations amongst Sobolev spaces imply that it suffices to show Theorem 1.2 for $1 + \sum_{i = 1}^m \alpha_i  < \gamma < n + 1$, 
so this is what we will assume.

The operator $T$ is a convolution operator taking $f$ to $f \ast \rho$ for some measure $\rho$. Then given any $\gamma$ satisfying
$1 + \sum_{i = 1}^m \alpha_i  < \gamma < n + 1$ , $(I - \Delta)^{-{\gamma \over 2}} Tf$ is given by $f \ast \sigma_{\gamma}$ where $\sigma_{\gamma}$
is the convolution of $\rho$ with the inverse Fourier transform of $(1 + |\xi|^2)^{-{\gamma \over 2}}$. Theorem 1.2 for such $\gamma$ will immediately follow from
Young's inequality once we prove the following lemma.

\begin{lemma}

For any $\gamma$ with $1 + \sum_{i = 1}^m \alpha_i  < \gamma < n + 1$, the measure $\sigma_{\gamma}$ is a function $h_{\gamma}(x)$ satisfying
\[ |h_{\gamma}(x)| \leq C_{\gamma}e^{-|x|} \tag{4.1}\]

\end{lemma}

\begin{proof} The inverse Fourier transform of $(1 + |\xi|^2)^{-{\gamma \over 2}}$ is the well known Bessel kernel $G_{\gamma}(x)$ which satisfies
 the following bounds for some $C_{\gamma}' > 0$.
\[|G_{\gamma}(x)| < C_{\gamma}'|x|^{-n - 1 + \gamma}\,\,\,\,\,\,\,(|x| < 1)\tag{4.2a}\]
\[|G_{\gamma}(x)| < C_{\gamma}'e^{-|x|}\,\,\,\,\,\,\,(|x| \geq  1)\tag{4.2b}\]
We refer to [AS] for more information about such estimates. Thus $h_{\gamma}(x) = \rho \ast G_{\gamma}(x)$, and we will show that $(4.1)$ is satisfied
using $(4.2a)-(4.2b)$.

First, note that since $\rho$ is a finite measure which we may assume is supported on $\{x: |x| < 1\}$, $(4.2b)$ implies that if 
$|x| > 2$ then $h_{\gamma}(x) = \rho \ast G_{\gamma}(x)$ satisfies $(4.1)$. Thus it suffices to consider showing 
$(4.1)$ is satisfied for $|x| < 2$. In other words, we must show that $h_{\gamma}(x)$ is bounded on $|x| < 2$. We will actually end out proving 
$h_{\gamma}(x)$ is bounded on all of $\R^{n+1}$.

Write $G_{\gamma}(x) = G_{\gamma}^1(x) + G_{\gamma}^2(x)$ where $G_{\gamma}^1(x) = 
G_{\gamma}(x)\chi_{\{x: |x| < 1\}}(x)$ and
$G_{\gamma}^2(x) = G_{\gamma}(x)\chi_{\{x: |x| \geq 1\}}(x)$. Note that by $(4.2a)$ the function $G_{\gamma}(x)$ is bounded on 
 $\{x: |x| \geq 1\}$. Hence $\rho \ast G_{\gamma}^2(x)$ is the convolution of a finite measure with a bounded function, and thus
is a bounded function. Hence to show Lemma 4.1 it suffices to show that $\rho \ast G_{\gamma}^1(x)$ is a bounded function. If as before ${\bf x}$ denotes
 $(x_1,...,x_n)$ and ${\bf t}$ denotes $(t_1,...,t_n)$, we have
\[\rho \ast G_{\gamma}^1(x) = \int_{\R^n} G_{\gamma}^1({\bf x} - {\bf t}, x_{n+1} - S({\bf t})) K({\bf t})\,d{\bf t}
\tag{4.3}\]
By $(4.2a)$ one has $|G_{\gamma}(x)| < C_{\gamma}'|x|^{-n - 1+ \gamma}$ for $|x| < 1$. Let $\psi(x)$ be a bump function on $\R$ that is nonnegative, even, decreasing on $x \geq 0$, equal to $1$ on $(-1,1)$ and
supported on $(-2,2)$. Then $|G_{\gamma}^1(x)| \leq C_{\gamma}'|x|^{-n - 1+ \gamma}\psi(|x|)$. 
Furthermore, by $(1.2a)$ one has $|K({\bf t})| \leq C\prod_{k=1}^m |{\bf t}_k|^{-\alpha_k}$. Substituting these bounds in $(4.3)$ gives the
following, where $X$ denotes the support of $K({\bf t})$.
 \[|\rho \ast G_{\gamma}^1(x)| \leq \int_X\psi({|(\bf x} - {\bf t}, x_{n+1} - S({\bf t}))|) \times |({\bf x} - {\bf t}, x_{n+1} - S({\bf t}))|^{-n - 1 + \gamma} \prod_{k=1}^m |{\bf t}_k|^{-\alpha_k}\,d{\bf t}
\tag{4.4}\]
Since $|x|^{-n  - 1 + \gamma}\psi(|x|)$ is decreasing in $|x|$, the integrand in $(4.4)$ is increased if we replace $x_{n+1} - S({\bf t})$ by
$0$. Hence we have
\[|\rho \ast G_{\gamma}^1(x)| \leq \int_X\psi(|{\bf x} - {\bf t}|)|{\bf x} - {\bf t}|^{-n - 1 + \gamma} \prod_{k=1}^m 
|{\bf t}_k|^{-\alpha_k}\,d{\bf t}\tag{4.5}\]
Since $X$ is compact, there are constants $B_k$ such that the right-hand side of $(4.5)$ is bounded by
\[\int_{\R^n} \psi(|{\bf x} - {\bf t}|)|{\bf x} - {\bf t}|^{-n - 1 + \gamma} \prod_{k=1}^m \psi(B_k |{\bf t}_k|)|{\bf t}_k|^{-\alpha_k}\,d{\bf t}\tag{4.6}\]
The assumed condition that $\gamma > 1 + \sum_{i =1}^m \alpha_i$ implies that $-n - 1 + \gamma > -n +  \sum_{i =1}^m \alpha_i = 
\sum_{i=1}^m (\alpha_i - l_i) = \sum_{i=1}^m -(l_i -\alpha_i)$. The assumed condition that $\gamma < n + 1$ simply means that $-n - 1 + \gamma < 0$.
Hence we may write $-n - 1 + \gamma = \sum_{k=1}^m -\beta_k$ where each $\beta_k$ satisfies $0 < \beta_k < (l_k - \alpha_k)$. We rewrite $(4.6)$ as
\[\int_{\R^n} \psi(|{\bf x} - {\bf t}|)\bigg( \prod_{k=1}^m |{\bf x} - {\bf t}|^{-\beta_k}\bigg)\psi(B_k |{\bf t}_k|)|{\bf t}_k|^{-\alpha_k}\,d{\bf t}\tag{4.7}\]
Since the functions $|y|^{-{\beta_k}}$ are decreasing in $y$,  this is bounded by
\[\int_{\R^n} \psi(|{\bf x} - {\bf t}|) \bigg(\prod_{k=1}^m |{\bf x}_k - {\bf t}_k|^{-\beta_k}\bigg)\psi(B_k |{\bf t}_k|)|{\bf t}_k|^{-\alpha_k}\,d{\bf t}\tag{4.8}\]
Furthermore, there are bump functions $\psi_k$ such that
\[\psi(|{\bf x} - {\bf t}|) \leq \prod_{k =1}^m \psi_k(|{\bf x}_k - {\bf t}_k|)\tag{4.9}\]
Inserting this into $(4.8)$ provides an upper bound of
\[\prod_{k=1}^m \int_{\R^{l_i}} \psi_k(|{\bf x}_k - {\bf t}_k|)|{\bf x}_k - {\bf t}_k|^{-\beta_k}\psi(B_k |{\bf t}_k|)|{\bf t}_k|^{-\alpha_k}\,d{\bf t}_k\tag{4.10}\]
I claim that since $0 < \alpha_k + \beta_k < l_k $, each integral in the product $(4.10)$ is uniformly bounded in $x_k$. One way to see this is to view the 
integral as a convolution of two functions, one of whose Fourier transforms is bounded by $C(1 + |\xi|)^{-(l_k  - \beta_k)}$ an the other whose Fourier
transform is bounded by  $C(1 + |\xi|)^{-(l_k  - \alpha_k)}$. Hence the Fourier transform of the convolution is bounded by $C'(1 + |\xi|)^{-(2l_k  -
\alpha_k -  \beta_k)}$. Since $\alpha_k + \beta_k < l_k $, the exponent $2l_k - \alpha_k - \beta_k$ is greater than $l_k$, which means the Fourier transform
of the convolution is integrable. Hence the inverse Fourier transform of this Fourier transform is uniformly bounded. In other words, each integral 
in $(4.10)$ is uniformly bounded. Hence looking back to $(4.4)$ we see that $|\rho \ast G_{\gamma}^1(x)|$ is uniformly bounded in $x$, completing
the proof of Lemma 4.1.

\end{proof}

\section{The proof of Theorem 1.5}

\noindent {\bf Motivation.}

We now assume that the hypotheses of Theorem 1.5 are satisfied. Since we are assuming each $\alpha_i = 0$, by definition of $g$ we have
that $g = \min(a_0,l_1,...,l_m)$. Since each $l_i$ is a positive integer and we are assuming that $g \leq {1 \over \max(o(S),2)} < 1$, we must have 
$a_0 = g$. 

It suffices to prove Theorem 1.5 in the case where $s > 0$. To see why, suppose we know Theorem 1.5 for $s > 0$,  and $s \leq 0$ is such that we have
 boundedness theorem for some $({1 \over p}, {1 \over q}, s)$ above
the plane $P$. Then we get a contradiction by interpolating this result with a boundedness theorem for an $({1 \over p}, {1 \over q}, s)$ provided by 
Theorem 1.4 with $s > 0$ 
that is a small distance beneath the plane $P$; the result is a boundedness theorem for $({1 \over p}, {1 \over q}, s)$ above the plane $P$
but with $s > 0$, contradicting Theorem 1.5 for $s > 0$. So in the following argument we can always assume $s > 0$.

We will prove Theorem 1.5 by testing $T$ on approximations to characteristic functions of rectangular boxes defined as follows. Let $d$ be the Newton distance of $S$, which we recall is
given by the minimal $t$ for which $(t,...,t)$ is in the Newton polyhedron $N(S)$. Let $b = (b_1,...,b_n)$ be a vector of nonnegative
numbers for which the infimum of $b \cdot \alpha$ over all vertices $\alpha$ of $N(S)$ is given by $b \cdot (d,...,d) = d\sum_{i=1}^n b_i$. In other
words, we let $b$ 
be such that there is a supporting hyperplane of $N(S)$ with normal $b$ containing $(d,...,d)$. The boxes we will use to test $T$ will have dimensions
comparable to $r^{b_1} \times ... \times r^{b_n} \times r^{d\sum_{i=1}^n b_i}$ and we will let $r \rightarrow 0$. Here if some $b_i = 0$ then we replace $r^{b_i}$ with a constant dimension $c_i$ that is
stipulated to be sufficiently small for our arguments to work.

To help understand the significance of such rectangular boxes, we go back to $(1.5)$ and examine $S^*({\bf t})$ on the box $({1 \over 2}r^{b_1},r^{b_1}) 
\times ...\times ({1 \over 2} r^{b_n}, r^{b_n})$. Then $S^*({\bf t})$ is 
comparable in magnitude to  the largest term $r^{\alpha \cdot b}$ for $\alpha \in N(S)$, which in turn is given by $r^{(d,...,d) \cdot b} = 
r^{d \sum_{i=1}^n b_i}$. On the other hand the volume of the box is comparable to $r^{\sum_{i=1}^n b_i}$. Thus if $\epsilon > 0$ is such that
$S^*({\bf t}) \sim \epsilon$ on the box, then $r \sim \epsilon^{1 \over d \sum_{i=1}^n b_i}$ and the volume of the box is comparable to
 $r^{\sum_{i=1}^n b_i} \sim \epsilon^{1 \over d}$.
As mentioned earlier, by Varchenko's [V] and other papers, $a_0 = {1 \over d}$. Hence in terms of $(1.5)$, for any $\epsilon > 0$ the associated box
contains a large chunk of the points where $S^*({\bf t}) < \epsilon$, in the sense that we are off at most by a constant times a logarithmic factor.

By the nature of the arguments used to prove the $L^p$ to $L^p_{s}$ estimates, the above considerations imply that the boxes of the previous
paragraph are natural for testing $L^p$ to $L^p_{s}$ estimates. On the other hand, one has a lot of flexibility in testing the 
$L^{1 + \epsilon}$ to $L^{q}_{-\gamma - \epsilon}$ estimates for $\epsilon \rightarrow 0$ and $q \rightarrow \infty$, and effectively one can
interpolate between the two situations so that the boxes can also be used to show that one can never get an $L^p$ to $L^q_{s}$ 
estimate above the plane $P$ in Theorems 1.3-1.4.

\noindent {\bf The main argument.}

Let $\psi(x)$ be a nonnegative bump function on $\R$ supported on $(-2,2)$ such that $\psi(x) \leq 1$ with $\psi(x) = 1$ on $(-1,1)$. Let $\psi_1(x)$
be a nonzero Schwartz function on $\R$ whose Fourier transform is supported in $(1,2)$. Let $b_{n+1} = d\sum_{i=1}^n b_i$ and for $r, N> 0$
 let $f_{r,N}(x)$ be defined by
\[f_{r,N}(x) = \psi_1(r^{-b_{n+1}}x_{n+1})\prod_{i = 1}^n \psi(N r^{-b_i} x_i) \tag{5.1}\]
 As above we replace $r^{-b_i} x_i$ by $c_i^{-1}x_i$ for an appropriately small constant $c_i$ in the event that
$b_i = 0$.

Let $D^s$ denote the operator with Fourier multiplier $|\xi_{n+1}|^s$.
We will show that given $p$, $q$, $s$ with $1 < p, q < \infty$ such that $({1 \over p}, {1 \over q}, t)$ is on the plane $P$ of Theorem 1.5 for some 
$t < s$, then if $N$ is 
sufficiently large  we have the following estimate for some $\epsilon > 0$, for all sufficiently small $r > 0$.
\[||D^s (T f_{r,N})||_{L^q} / ||f_{r,N}||_{L^p} > C r^{-\epsilon} \tag{5.2}\]
 I claim that this suffices to prove that one cannot have an estimate of the form $||T f_{r,N}||_{L^q_s} \leq C||f_{r,N}||_{L^p}$. For if 
we did have such an estimate, we could compose it with the operator with multiplier of the form $\phi(r^{b_{n+1}}\xi_{n+1})
|\xi_{n+1}|^s / (1 + |\xi|^2)^{s \over 2}$, where $\phi$ is a bump function supported on $(1/2,3)$ and equal to 1 on $(1,2)$. This multiplier is uniformly bounded in $r$ on $L^q$
by the Marcinkiewicz multiplier theorem (see p. 108 of [Ste]). Since the Fourier transform of ${\psi_1}$ is supported on $(1,2)$, this composition
 acting on $T f_{r,N}$ is just  $D^s T f_{r,N}$ and we 
obtain that $||D^s T f_{r,N}||_{L^q} \leq C'||f_{r,N}||_{L^p}$. This contradicts $(5.2)$ as
$r \rightarrow 0$. Hence it will suffice to prove $(5.2)$.

Next, by the translation invariance of $T$ we have $D^s (T f_{r,N}) = T(D^s f_{r,N})$ and we examine the effect of $D^s$ on $f_{r,N}$. Since $\psi_1$
has Fourier transform supported on the interval $(1,2)$, $D^s f_{r,N}(x)$ is of the form
\[D^s f_{r,N}(x) = r^{-sb_{n+1}}\Psi_s(r^{-b_{n+1}}x_{n+1})\prod_{i = 1}^n \psi(N r^{-b_i} x_i) \tag{5.3}\]
Here $\Psi_s (x_{n+1})$ is of the same form as $\psi_1(x_{n+1})$, but has been modified due to the multiplier. Let $p$ be such that $\Psi_s(p) \neq 0$.
Then there are some $\epsilon_0, \delta_0 > 0$ such that $|\Psi_s(p')|> \epsilon_0$ and $|\arg(\Psi_s(p)) - \arg(\Psi_s(p'))| < {\pi \over 4}$ 
when $|p' - p| < 2\delta_0$.

We examine $D^s(T f_{r,N})(x) =
T(D^s f_{r,N} (x))$ for a fixed $x$ such that $|x_{n+1} - r^{b_{n+1}}p| < \delta_0 r^{b_{n+1}}$ and $|x_i| < {1 \over 2N}r^{b_i}$ for each $i \leq n$.
Observe that $T(D^s f_{r,N}(x))$ is the average of $D^s f_{r,N}(x')$ in $x'$ over a surface centered at $x$, which by the assumptions of Theorem 1.5 is
weighted by a nonnegative function which is bounded below by some $C_1 > 0$  near $x$ . Since $|x_i| < {1 \over 2N}r^{b_i}$ for 
each $i \leq n$, the portion of the average corresponding to $x'$ on the surface with $|x_i' - x_i| < {1 \over 2N}r^{b_i}$ for $1 \leq i \leq n$ will be such that  
the $\psi(N r^{-b_i} x_i')$ factors in $(5.3)$ will all be $1$. 
 
Next, by Lemma 2.1 of [G2] there is a constant $C_0$ such that $|S({\bf t}) | \leq C_0|S^*({\bf t})|$ for all $t$ in a sufficiently small neighborhood of the origin.
 Thus when $|t_i| < {1 \over 2N}r^{b_i}$ for each $i$, one has 
\[|S({\bf t}) | \leq C_0|S^*({\bf t})|\]
\[ \leq C_0 \sup_{\{t: |t_i|< {1 \over 2N}r^{b_i}{\rm\,\,for\,\,all\,\,}i \}} |S^*({\bf t})|\]
\[ \leq C_0 {1 \over 2N}\sup_{\{t: |t_i|< r^{b_i}{\rm\,\,for\,\,all\,\,}i\}} |S^*({\bf t})| \tag{5.4}\]
As described in the motivation section above, $\sup_{\{t: |t_i|< r^{b_i}{\rm\,\,for\,\,all\,\,}i\}} |S^*({\bf t})| \leq C_1r^{d\sum_{i=1}^n b_i}
 = C_1r^{b_{n+1}}$.
As a result, if $N$ is large enough, we can ensure that if $|t_i| < {1 \over 2N}r^{b_i}$ for each $i$ then we have
\[|S({\bf t})| \leq \delta_0 r^{b_{n+1}}\tag{5.5}\]
 Thus when averaging $D^s f_{r,N}(x')$ over $x'$ on the surface centered at $x$, the portion corresponding to 
where $|x_i' - x_i| < {1 \over 2N}r^{b_i}$ for all $i$ will have always have its final coordinate satisfying $|x_{n+1}' - x_{n+1}| < \delta_0r^{b_{n+1}}$,
so that $|x_{n+1}'- r^{b_{n+1}}p| \leq |x_{n+1}' - x_{n+1}| + |x_{n+1} -  r^{b_{n+1}}p| < 2\delta_0r^{b_{n+1}}$.
Hence by definition of $\delta_0$, $|\Psi_s(r^{-b_{n+1}}x_{n+1}')| > \epsilon_0$ and the argument of $\Psi_s(r^{-b_{n+1}}x_{n+1}')$ is within ${\pi \over 4}$ of that of
$\Psi_s(r^{-b_{n+1}}x_{n+1})$.

Therefore, when viewing $D^s(T f_{r,N})(x) = T(D^s f_{r,N}(x))$ as the average of $D^s f_{r,N}(x')$ along a surface centered at $x$, the portion where 
$|x_i' - x_i| < {1 \over 2N}r^{b_i}$
for each $i$ corresponds to points where the $\psi(N r^{-b_i} x_i')$ factors in $(5.3)$ are all $1$, where $|\Psi_s(r^{-b_{n+1}}x_{n+1}')| > \epsilon_0$,
and where the argument of $\Psi_s(r^{-b_{n+1}}x_{n+1}')$ is within ${\pi \over 4}$ of that of $\Psi_s(r^{-b_{n+1}}x_{n+1})$.
As a result, we have $|D^s f_{r,N}(x')| > \epsilon_0 r^{-sb_{n+1}}$ at such points. 
 
Hence when averaging $D^s f_{r,N}(x')$ over the $x'$ with $|x_i' - x_i| < {1 \over 2N}r^{b_i}$ for all $1 \leq i \leq n$, one obtains a contribution to this
 average of absolute value at least
$({1 \over N})^n{\epsilon_0 \over 2}  C_1 r^{(\sum_{i=1}^n b_i) - sb_{n+1}}$ coming from these points. While it is true that there is also a contribution from other $x'$ where $|x_i' - x_i| < {5 \over 2N}r^{b_i}$
for each $i$, if $N$ is large enough we will still have $|\Psi_s(r^{-b_{n+1}}x_{n+1}')| > \epsilon_0$ and $|\arg(\Psi_s(r^{-b_{n+1}}x_{n+1}'))
- \arg(\Psi_s(r^{-b_{n+1}}x_{n+1}))| < {\pi \over 4}$, so this contribution 
will only amplify the previous contribution. In summary, if $N$ is large enough, there is a constant $\epsilon_1 > 0$ (which can depend on $s$ and $N$) 
such that  if 
 $|x_{n+1} - r^{b_{n+1}}p| < \delta_0 r^{b_{n+1}}$ and $|x_i| < {1 \over 2N}r^{b_i}$ for each $i \leq n$  then
\[|D^s(T f_{r,N})(x)| > \epsilon_1  r^{(\sum_{i=1}^n b_i) -sb_{n+1}}\tag{5.6}\]
The $L^q$ norm of $D^s(T f_{r,N})(x)$ is at least the $L^q$ norm of $D^s(T f_{r,N})(x)$ as a function on the set of points where  $|x_{n+1} - r^{b_{n+1}}p| < \delta_0 r^{b_{n+1}}$ and $|x_i| < {1 \over 2N}r^{b_i}$ for each $i \leq n$, so we have
\[||D^s(T f_{r,N})||_q > \epsilon_2 r^{(\sum_{i=1}^n b_i) -sb_{n+1}} \times r^{\sum_{i=1}^{n+1} b_i \over q} \tag{5.7a}\]
Recalling that $b_{n+1} = d\sum_{i = 1}^n b_i$, where $d$ is the Newton distance of $S$, $(5.7a)$ can be rewritten as 
\[||D^s(T f_{r,N})||_q > \epsilon_2 r^{(\sum_{i=1}^n b_i)(1 - sd + {1 + d\over q})}\tag{5.7b}\]
On the other hand, the $L^p$ norm of $f_{r,N}(x)$ satisfies
\[||f_{r,N}||_p \sim r^{\sum_{i=1}^{n+1} b_i \over p} = r^{(\sum_{i=1}^{n} b_i){(1 + d)\over p}} \tag{5.8}\]
Thus we have
\[||D^s(T f_{r,N})||_q/ ||f_{r,N}||_p > \epsilon_3 r^{(\sum_{i=1}^{n}b_i)(1 - sd + {d+1 \over q} - {d + 1 \over p})}\tag{5.9}\]
The exponent here is negative exactly when $s > {1 \over d} + {d + 1 \over d}{1 \over q} - {d + 1 \over d}{1 \over p}$. Confirming that this is in fact does
correspond to the
equation of the plane $P$, we observe that since each $\alpha_i = 0$ we have $k = 1$ in the definition of $P$. Furthermore, as described in the 
beginning of the motivation section above, we have $a_0 = g$, and as mentioned before, by [V] or [G2] 
we also have $a_0 = {1 \over d}$. So $g = {1 \over d}$ here.
Hence the plane $z = {1 \over d} + {d + 1 \over d}y  - {d + 1 \over d}x$ is the same as the plane $(g + k)(x - y) + z = g$. Therefore the condition
 $s > {1 \over d} + {d + 1 \over d}{1 \over q} - {d + 1 \over d}{1 \over p}$ is equivalent to $({1 \over p}, {1 \over q}, s)$ lying above $P$ as needed.
 This completes the proof of Theorem 1.5. \qed

\section {References.}

\noindent [AS] N. Aronszajn, K. T. Smith, {\it Theory of Bessel potentials I}, Ann. Inst. Fourier {\bf 11} (1961) 385-475. \parskip = 4pt\baselineskip = 3pt

\noindent [Cu] S. Cuccagna, {\it  Sobolev estimates for fractional and singular Radon transforms}, J. Funct. Anal. {\bf 139} (1996), no. 1, 94-118.

\noindent [DZ] S. Dendrinos, E. Zimmermann, {\it On $L^p$-improving for averages associated to mixed homogeneous polynomial hypersurfaces in $\R^3$},
 J. Anal. Math. {\bf 138} (2019), no. 2, 563-595. 

\noindent [FGU1] E. Ferreyra, T. Godoy, M. Urciuolo, {\it Boundedness properties of some convolution operators with singular measures}, Math. Z. 
{\bf 225} (1997), no. 4, 611-624.

\noindent [FGU2] E. Ferreyra, T. Godoy, M. Urciuolo, {\it Sharp $L^p$-$L^q$ estimates for singular fractional integral operators.}
Math. Scand. {\bf 84} (1999), no. 2, 213-230. 

\noindent [G1] M. Greenblatt, {\it $L^p$ Sobolev regularity of averaging operators over hypersurfaces and the Newton polyhedron}, J. Funct. Anal. 
{\bf 276} (2019), no. 5, 1510-1527.

\noindent [G2] M. Greenblatt, {\it Oscillatory integral decay, sublevel set growth, and the Newton 
polyhedron}, Math. Annalen {\bf 346} (2010) no. 4, 857-890. 

\noindent [G3] M. Greenblatt, {\it Smooth and singular maximal averages over 2D hypersurfaces  and associated Radon transforms}, submitted.

\noindent [G4] M. Greenblatt, {\it Uniform bounds for Fourier transforms of surface measures in $\R^3$ with nonsmooth density},
Trans. Amer. Math. Soc. {\bf 368} (2016), no. 9, 6601-6625.

\noindent [G5] M. Greenblatt, {\it An analogue to a theorem of Fefferman and Phong for averaging operators
along curves with singular fractional integral kernel}, Geom. Funct. Anal. {\bf 17} (2007), no. 4, 1106-1138.

\noindent [Gr] P. T. Gressman, {\it Uniform Sublevel Radon-like Inequalities}, J. Geom. Anal. {\bf 23} (2013), no. 2,
611-652.

\noindent [ISaS] A. Iosevich, E. Sawyer, and A. Seeger, {\it On averaging operators associated with convex hypersurfaces of finite type}, 
J. Anal. Math. {\bf 79} (1999), 159-187.

\noindent [Ka1] V. N. Karpushkin, {\it A theorem concerning uniform estimates of oscillatory integrals when
the phase is a function of two variables}, J. Soviet Math. {\bf 35} (1986), 2809-2826.

\noindent [Ka2] V. N. Karpushkin, {\it Uniform estimates of oscillatory integrals with parabolic or 
hyperbolic phases}, J. Soviet Math. {\bf 33} (1986), 1159-1188.

\noindent [O] D. M. Oberlin, {\it Convolution with measures on hypersurfaces}, Math. Proc. Camb. Phil. Soc.
{\bf 129} (2000), no. 3, 517-526.

\noindent [S] A. Seeger, {\it Radon transforms and finite type conditions}, J. Amer. Math. Soc. {\bf 11} (1998), no. 4, 869-897.

\noindent [St] B. Street, {\it Sobolev spaces associated to singular and fractional Radon transforms}, Rev. Mat. Iberoam. {\bf 33} (2017), no. 2, 633-748.

\noindent [Ste] E. M. Stein, {\it Singular integrals and differentiability properties of functions}, Princeton Mathematical Series, No. 30, Princeton University Press, Princeton, N.J., 1970.

\noindent [V] A. N. Varchenko, {\it Newton polyhedra and estimates of oscillatory integrals}, Functional 
Anal. Appl. {\bf 18} (1976), no. 3, 175-196.

\end{document}